\documentclass{article}[12pt]
\usepackage{indentfirst, latexsym, bm, amssymb, amsmath}

\begin{document}
\newtheorem{theorem}{Theorem}[section]
\newtheorem{corollary}[theorem]{Corollary}
\newtheorem{definition}[theorem]{Definition}
\newtheorem{conjecture}[theorem]{Conjecture}
\newtheorem{question}[theorem]{Question}
\newtheorem{lemma}[theorem]{Lemma}
\newtheorem{proposition}[theorem]{Proposition}
\newtheorem{example}[theorem]{Example}
\newenvironment{proof}{\noindent {\bf
Proof.}}{\rule{3mm}{3mm}\par\medskip}
\newcommand{\remark}{\medskip\par\noindent {\bf Remark.~~}}
\newcommand{\pp}{{\it p.}}
\newcommand{\de}{\em}

\newcommand{\JEC}{{\it Europ. J. Combinatorics},  }
\newcommand{\JCTB}{{\it J. Combin. Theory Ser. B.}, }
\newcommand{\JCT}{{\it J. Combin. Theory}, }
\newcommand{\JGT}{{\it J. Graph Theory}, }
\newcommand{\ComHung}{{\it Combinatorica}, }
\newcommand{\DM}{{\it Discrete Math.}, }
\newcommand{\ARS}{{\it Ars Combin.}, }
\newcommand{\SIAMDM}{{\it SIAM J. Discrete Math.}, }
\newcommand{\SIAMADM}{{\it SIAM J. Algebraic Discrete Methods}, }
\newcommand{\SIAMC}{{\it SIAM J. Comput.}, }
\newcommand{\ConAMS}{{\it Contemp. Math. AMS}, }
\newcommand{\TransAMS}{{\it Trans. Amer. Math. Soc.}, }
\newcommand{\AnDM}{{\it Ann. Discrete Math.}, }
\newcommand{\NBS}{{\it J. Res. Nat. Bur. Standards} {\rm B}, }
\newcommand{\ConNum}{{\it Congr. Numer.}, }
\newcommand{\CJM}{{\it Canad. J. Math.}, }
\newcommand{\JLMS}{{\it J. London Math. Soc.}, }
\newcommand{\PLMS}{{\it Proc. London Math. Soc.}, }
\newcommand{\PAMS}{{\it Proc. Amer. Math. Soc.}, }
\newcommand{\JCMCC}{{\it J. Combin. Math. Combin. Comput.}, }
\newcommand{\GC}{{\it Graphs Combin.}, }

\title{ Tur\'{a}n numbers for disjoint paths
\thanks{This work is supported by the National Natural Science Foundation of China (Nos.11531001 and 11271256),  the Joint NSFC-ISF Research Program (jointly funded by the National Natural Science Foundation of China and the Israel Science Foundation (No. 11561141001)),  Innovation Program of Shanghai Municipal Education Commission (No. 14ZZ016) and Specialized Research Fund for the Doctoral Program of Higher Education (No.20130073110075).
\newline \indent $^{\dagger}$Correspondent author:
Xiao-Dong Zhang (Email: xiaodong@sjtu.edu.cn)}}
\author{ Long-Tu Yuan  and Xiao-Dong Zhang$^{\dagger}$   \\
{\small School of Mathematical Sciences, MOE-LSC, SHL-MAC
}\\
{\small Shanghai Jiao Tong University} \\
{\small  800 Dongchuan Road, Shanghai, 200240, P.R. China}\\
{\small Email: yuanlongtu@sjtu.edu.cn, xiaodong@sjtu.edu.cn }}
\date{}
\maketitle
\begin{abstract}
  The Tur\'{a}n number of a graph $H$, $ex(n,H)$, is the maximum number of edges in any graph of order $n$  which does not contain $H$ as a subgraph. Lidick\'{y}, Liu and  Palmer  determined $ex(n, F_m)$ for $n$ sufficiently large and proved  that the extremal graph is unique, where $F_m$ is disjoint paths of $P_{k_1}, \ldots, P_{k_m}$ [Lidick\'{y},~B., Liu,~H. and  Palmer,~C. (2013).  On the Tur\'{a}n number of forests. {\em   Electron. J. Combin. } {\bf 20(2)} Paper 62, 13 pp]. In this paper,  by mean  of a different approach, we  determine  $ex(n, F_m)$ for  all  integers $n$ with minor conditions, which extends their partial results. Furthermore, we  partly confirm the conjecture proposed by Bushaw and Kettle  for $ex(n, k\cdot P_l)$ [Bushaw,~N. and Kettle,~N. (2011) Tur\'{a}n numbers of multiple paths and equibipartite forests. {\em Combin. Probab. Comput.} {\bf 20} 837-853]. Moreover, we show that there exist two family graphs $F_m$ and $F_m^{\prime}$ such that $ex(n, F_m)=ex(n, F_m^{\prime})$ for all integers $n$, which is related to an old problem of Erd\H{o}s and Simonovits.
\end{abstract}


\vskip 0.5cm
\section{Introduction}
The graphs considered in this paper are finite, undirected, and simple (no loops or multiple edges). Let $G=(V(G), E(G))$ be a simple graph,  where $V(G)$ is the vertex set and $E(G)$ is the edge set with size $e(G)$. Let $G$ and $H$ be two disjoint graphs. Denote by $G\bigcup H$ the disjoint union of $G$ and $H$ and by $k\cdot G$ the disjoint union of $k$ copies of a graph $G$. Denote by $G+H$ the graph obtained from $G\bigcup H$ by joining each vertex of $G$ to each vertex of $H$. If $S\subseteq V(G)$, the subgraph of $G$ induced by $S$ is denoted by $G[S]$. Moreover, Denote by $P_{k}$ a path on $k$ vertices, $K_{n}$ the completed graph with $n$ vertices, $\overline{G}$ the complement graph of $G$.

The {\it Tur\'{a}n number} of a graph $H$, $ex(n,H)$, is the maximum number of edges in a graph $G$ of order $n$  which does not contain $H$ as a subgraph. Denote by $Ex(n,H)$ a graph on $n$ vertices with $ex(n,H)$ edges  which does not contain $H$ as a subgraph and call this graph an {\it extremal graph} for $H$. Moreover,  Denote by $ex_{con}(n,H)$  the maximum number of edges in a connected graph $G$ of order $n$  which does not contain $H$ as a subgraph, and denote by $Ex_{con}(n,H)$ a  connected graph on $n$ vertices with $ex_{con}(n,H)$ edges  which does not contain $H$ as a subgraph.  Often there are several extremal graphs. In 1941, Tur\'{a}n \cite{turan1941} proved that the extremal graph  which does not contain $K_r$ as a subgraph is the complete $(r-1)$-partite graph on $n$ vertices which is balanced, in that the part sizes are as equal as possible (any two sizes differ by at most $1$). This balanced complete $(r-1)$-partite graph on $n$ vertices is the Tur\'{a}n graph $T_{r-1}(n)$. On the other hand, for sparse graphs,
 Erd\H{o}s and Gallai \cite{erdHos1959maximal} in 1959 proved the following well known result.

\begin{theorem}\cite{erdHos1959maximal}\label{Pathk1}
 Let $G$ be a graph with $n$ vertices. If $G$ does not contain a path with $k$ vertices and $n\ge k\ge 2$, then $e(G)\leq \frac{1}{2}(k-2)n$ with equality if and only if $n=(k-1)t$ and $G=\bigcup_{i=1}^{t}K_{k-1}$.
\end{theorem}
   It follows from the above theorem that $ex(n, P_k)=\frac{1}{2}(k-2)n$ for $n=(k-1)t$, and $ex(n, P_k) $ is not determined for $k-1$ not being a factor of $n$.  Later Faudree and Schelp \cite{Faudree1975} extended the above result for all integers $n$ and $k$.

For convenience, we first introduce the following symbols.

\begin{definition}
Let $n\geq m\geq l\ge 3$ be given three positive integers. Then $n$ can be written as $n=(m-1)+t(l-1)+r$, where $t\geq 0$ and $ 0\leq r < l-1.$  Denote by
 $$[n,m,l]\equiv {m-1 \choose 2}+t{l-1 \choose 2}+{r \choose 2}$$
  and
  $$[n,m]\equiv {\lfloor\frac{m}{2}\rfloor-1 \choose 2}+\lfloor\frac{m-2}{2}\rfloor\left(n-\lfloor\frac{m}{2}\rfloor+1\right).$$ Moreover, if  $n\leq m-1$, denote by $$[n,m,l]\equiv{n \choose 2}.$$
\end{definition}
Let $n\geq m\geq l\ge 3$. If $G_{1}=K_{m-1}\bigcup t\cdot K_{l-1}\bigcup K_{r}$ and $G_{2}=\overline{K}_{n-\lfloor\frac{m}{2}\rfloor+1}+K_{\lfloor\frac{m}{2}\rfloor-1}$, then $$e(G_1)=[n,m,l],\ \  e(G_2)=[n,m].$$

\begin{theorem}\cite{Faudree1975}\label{Pathk2}
 Let $G$ be a graph with $n\ge k$ vertices, if $G$ does not contain a path with $k$ vertices, then $e(G)\le [n,k,k]$ with equality if and only if $G$ is either $G=
 (\bigcup_{i=1}^{t}K_{k-1})\bigcup K_{r}$ or $G=
 (\bigcup_{i=1}^{t-s-1}K_{k-1})\bigcup(K_{\frac{k-2}{2}}+\overline{K}_{\frac{k}{2}+s(k-1)+r})$ for some $s$, $0\leq s<t$, when $k$ is even, $t>0$, and $r=\frac{k}{2}$ or $\frac{k-2}{2}$, where $n=(k-1)t+r$ and $0\le r<k-1$.
\end{theorem}
In other words, $ex(n, P_k)$ has been determined  for all integers $n> k$ and all extremal graphs has also been characterized. For connected graphs, Kopylov \cite{Kopylov1977}, in 1977, determined $ex_{con}(n, P_k)$. In 2008, Balister, Gy\H{o}ri, Lehel and Schelp \cite{Balister2008} used a different approach and determined $ex_{con}(n, P_k)$ and characterized all extremal graphs.
\begin{theorem}\cite{Balister2008,Kopylov1977}\label{conenctedpath}
Let $G$ be a connected graph on $n\geq k\geq4$ vertices which does not contain a path with $k$ vertices. Then
$$e(G)\leq \max\left\{{k-2 \choose 2}+(n-k+2), [n,k]+i-1\right\}.$$
$$(i=2 \mbox{ when } k \mbox{ is odd, } i=1 \mbox{ when } k \mbox{ is even})$$
 Further, the equality occurs if and only if $G$ is either $$(K_{k-3}\bigcup\overline{K}_{n-k+2})+K_{1}$$ or $$(K_{i}\bigcup\overline{K}_{n-\lfloor\frac{k+1}{2}\rfloor})+K_{\lfloor\frac{k}{2}\rfloor-1}.$$
\end{theorem}

\begin{remark}
A simple calculation shows that for $ k>5$, if $k$ is even, the extremal graphs are
$$\begin{array}{llll}
(K_{k-3}\bigcup\overline{K}_{n-k+2})+K_{1} &\mbox{for}\ \ n\leq \frac{5k-10}{4};\\ (K_{1}\bigcup\overline{K}_{n-\lfloor\frac{k+1}{2}\rfloor})+K_{\lfloor\frac{k}{2}\rfloor-1} &\mbox{for}\ \ n\geq\frac{5k-10}{4}. \end{array}$$
If $k$ is odd, the extremal graphs are
$$\begin{array}{llll}
(K_{k-3}\bigcup\overline{K}_{n-k+2})+K_{1} &\mbox{for}\ \ n\leq \frac{5k-7}{4};\\ (K_{2}\bigcup\overline{K}_{n-\lfloor\frac{k+1}{2}\rfloor})+K_{\lfloor\frac{k}{2}\rfloor-1} &\mbox{for}\ \ n\geq \frac{5k-7}{4}. \end{array}$$
\end{remark}

 In 1962, Erd\H{o}s \cite{erdHos1962} first studied on the Tur\'{a}n number of $k\cdot K_3$. Later, Moon \cite{moon1968} (only when $r-1$ divides $n-k+1$) and Simonovits \cite{simonovits1968} showed that $K_{k-1}+T_{r-1}(n-k+1)$ is the unique extremal graph which does not contain $k\cdot K_r$ for $n$ sufficiently large. In 2011, Bushaw and Kettle \cite{Bushwa2011} determined $ex(n,k\cdot P_{l})$ for $n$ sufficiently large.

 \begin{theorem}\cite{Bushwa2011}\label{kpl}
If $k\ge 2, l\ge 4$ and $n\ge 2l+2kl(\lceil\frac{l}{2}\rceil+1) {l \choose \lfloor\frac{l}{2}\rfloor}$, then
$$ex(n,k\cdot P_{l})=\left[n,k\lfloor\frac{l}{2}\rfloor\right]+c_l, $$
where $c_l=1$ if $l$ is odd, and $c_l=0$ if $l$ is even.
\end{theorem}

Further, their proof shows that their construction is optimal for $n = \Omega(kl^{\frac{3}{2}}2^l)$. Hence Bushaw and Kettle conjectured
that their construction is optimal for $n = \Omega(kl).$ In other words, they conjectured that the above theorem holds for $n = \Omega(kl).$
 Recently, Lidick\'{y}, Liu and Palmer \cite{lidicky2013} extended  Bushaw and Kettle's result and  determined $ex(n,F_m)$ for $n$ sufficiently large, where  $F_{m}=P_{k_1}\bigcup P_{k_2}\bigcup \ldots \bigcup P_{k_m} $ and $k_1\geq k_2\geq \ldots \geq k_m.$

\begin{theorem}\cite{lidicky2013}\label{Fm}
Let $F_m=P_{k_1}\bigcup P_{k_2}\bigcup \ldots \bigcup P_{k_m}$ and $k_1\geq k_2\geq \ldots \geq k_m.$ If at least one $k_i$ is not $3$, then for $n$ sufficiently large,
$$ex(n,F_m)=\left[n,\sum_{i=1}^{m}\lfloor\frac{k_i}{2}\rfloor\right]+c,$$
where $c=1$ if all $k_i$ are odd and $c=0$ otherwise. Moreover, the extremal graph is unique.
\end{theorem}
However, They do not consider the Tur\'{a}n number $ex(n, F_m)$ for small $n$. If $k_1=k_2=\ldots=k_m=3$, Gorgol \cite{gorgol} determined $ex(n,2\cdot P_3)$ and $ex(n,3\cdot P_3)$. Further Bushaw and Kettle \cite{Bushwa2011} determined $ex(n,k\cdot P_3)$ for $n\ge 7k$, and the extremal graphs  are unique $K_{k-1}+M_{n-k+1}$. Later, Yuan and Zhang \cite{Yuan} determined the value $ex(n,k\cdot P_3)$ for all $n$ and all extremal graphs which  are $K_{3k-1}\bigcup M_{n-3k+1}$ and $K_{k-1}+M_{n-k+1}$.

  Further Erd\H{o}s and Simonovits \cite{erdHos1976}  (see also \cite{bollobas1979}, chapter 6, problem 41.) asked that if $F_1$ and $F_2$ are two bipartite graphs, Giving conditions on $F_1$ and $F_2$ ensuring that $ex(n,F_1)=ex(n,F_2)$. In addition, it is nature to ask what is the Tur\'{a}n number of disjoint union of paths, cycles in hypergraphs (see \cite{bushwa2014, gu}). Another similar problem is the Erd\H{o}s's matching conjecture \cite{erdHos1965} which is a very difficult problem of Tur\'{a}n problems for expansions \cite{Mubayi2015}, especially when $n$ is small. The readers may be referred to \cite{Frankl2012, Frankl2013, H.huang2012}. For more information about Tur\'{a}n number problems, recently, there are two excellent surveys \cite{Furedi2013, keevash2011}.

Motivated by the results of \cite{Bushwa2011,lidicky2013} and other related results, we study the Tur\'{a}n number $ex(n, F_m)$ for all integers $n$, especial for $n$ small. Our main results in this paper can be stated as follows.
\begin{theorem}\label{main}
Let $k_1\geq k_2\geq\ldots\geq k_{m}\geq3$ and $n\ge \sum_{i=1}^mk_i$. If there is at most one odd in $\{k_1,k_2,\ldots,k_{m}\}$, then
 $$ex(n,F_m)= \max\left\{[n,k_1,k_1],[n,k_1+k_2,k_2],\ldots,[n,\sum_{i=1}^{m}k_{i},k_{m}],[n,\sum_{i=1}^{m}k_{i}]\right\}.$$
Moreover, if $k_1,k_2,\ldots,k_{m}$ are  even, then the extremal graphs are characterized.
\end{theorem}
 If there are two odds in $\{k_1, \ldots, k_m\}$, we have the following partial results.
\begin{theorem}\label{p2l+1,3}
Let $n\ge 2l+4$. Then
\begin{eqnarray*}
&&ex(n,P_{2l+1}\bigcup P_{3})=\max\left\{[n,2l+1,2l+1],[n,2l+4,3],[n,l]+1\right\}.
\end{eqnarray*}
Moreover, the extremal graphs are
$$Ex(n,P_{2l+1}),K_{2l+3}\bigcup M_{n-2l-3} \mbox{ and } K_l+\left(K_2\bigcup \overline{K}_{n-l-2}\right),$$
where $M_{n-2l-3}$ is a maximum matching with $n-2l-3$ vertices.
\end{theorem}

\begin{theorem}\label{p5,p5}
Let $n\ge 10$. Then
$$ex(n,P_{5}\bigcup P_{5})= \max\left\{[n,10,5],3n-5\right\}.$$
 Moreover, the extremal graphs are $K_9\bigcup Ex(n,P_5)$ and $K_3+(K_2\bigcup \overline{K}_{n-5})$.
\end{theorem}

The rest of this paper is organized as follows. In Section 2, several technical Observations and Lemmas  are obtained.  In Section 3, the proofs of Theorem~\ref{main} and corollaries are presented.  Further, we partly confirm Bushaw and Kettle's conjecture and present two family graphs $F_m$ and $F_m^{\prime}$ such that $ex(n, F_m)=ex(n, F_m^{\prime})$ for all $n$.  In Sections 4 and 5, the proofs of Theorems~\ref{p2l+1,3} and \ref{p5,p5} are presented, respectively. In Section 6, a conjecture is proposed for the conclusion.

\section{Several Observations and Lemmas}

\subsection{Several Observations}
In order to prove Lemmas and main results, we need the following Observations. Let $k_1\geq k_2\geq k_m\ge 3$ be three positive integers with at most one odd. The following observations can be proved with the help of the extremal graphs of Theorems~\ref{Pathk2}, Theorem~\ref{conenctedpath} and some calculations, which are given in Appendix A.

{\bf Observation 1:} Let $n\ge k_1+k_m$. Then
\begin{eqnarray*}
&&\max\left\{{k_1+k_m-2 \choose 2}+n-k_1-k_m+2,[n,k_1+k_m]\right\}\\&\leq &\max\left\{[n,k_1+k_m,k_m],[n,k_1+k_m]\right\}.
\end{eqnarray*}

{\bf Observation 2:} Let $n_1\ge k_1$. Then $$[n_1,k_1+k_m,k_m]+[n_2,k_m,k_m]\leq[n_1+n_2,k_1+k_m,k_m].$$
 Moreover, if $n_1=k_1+t_1(k_m-1)+r_1,n_2=t_2(k_m-1)+r_2$, where $0\le r_1< k_m-1$ and $0\le r_2< k_m-1,$ then Observation 2 becomes equality  only when $r_1=0$ or $r_2=0$.

{\bf Observation 3:} Let $n_1 \ge k_1$ and $n_2\ge k_2$. Then $$[n_1,k_1+k_m,k_m]+[n_2,k_2+k_m,k_m]<[n_1+n_2,k_1+k_2+k_m,k_m].$$

{\bf Observation 4:} Let $n_1\ge k_1+k_m $ and $n_2\ge k_2+k_m$. Then $$[n_1,k_1+k_m]+[n_2,k_2+k_m]<[n_1+n_2,k_1+k_2+k_m].$$

{\bf Observation 5:} Let $n_1\ge k_1+k_m$. Then $$[n_1,k_1+k_m]+[n_2,k_m,k_m]<[n_1+n_2,k_1+k_m].$$

{\bf Observation 6:} Let $n_1\ge k_1+k_m$ and $n_2\ge k_2$.  Then $$[n_1,k_1+k_m]+[n_2,k_2+k_m,k_m]<[n_1+n_2,k_1+k_2+k_m].$$

{\bf Observation 7:} Let $n_1\ge k_1,n_2\ge k_2+k_m$. Then
 \begin{eqnarray*}
&& [n_1,k_1+k_m,k_m]+[n_2,k_2+k_m]\\&< &\max\left\{[n_1+n_2,k_1+k_2+k_m,k_m],[n_1+n_2,k_1+k_2+k_m]\right\}.
\end{eqnarray*}

\subsection{Several Lemmas}

\begin{lemma}\label{connected}
Let $G$ be a connected graph with $n$ vertices and $F_m=P_{k_1}\bigcup \ldots \bigcup P_{k_m}$, where $k_1\geq k_2\geq\ldots\geq k_{m}\geq3,$ $k=\sum_{i=1}^{m}k_{i}$, and $ m\geq2$.

 (1) If there are all even in $\{k_1,k_2,\ldots,k_m\}$, then $ex_{con}(n,F_m)=ex_{con}(n,P_{k})$. Moreover, the extremal graph is $Ex_{con}(n,P_k)$.

 (2) If there is exact one odd in  $\{k_1,k_2,\ldots,k_m\}$ with $k_m>3$, then $ex_{con}(n,F_m)=\max\{{k-2 \choose 2}+(n-k+2),[n,k]\}$.

 (3) If there are all even in $\{k_1,k_2,\ldots,k_{m-1}\}$ with $k_m=3$, then $ex_{con}(n,F_m)\leq ex_{con}(n,P_{k})-1=\max\{{k-2 \choose 2}+(n-k+2)-1,[n,k]\}$.
\end{lemma}
\begin{proof}
(1) If there are all even in $\{k_1,k_2,\ldots,k_m\}$,  by Theorem~\ref{conenctedpath}, it is easy to see that  $Ex_{con}(n,P_{k})$  contains no $F_m$, which implies that $ex(n, F_m)\ge ex_{con}(n,P_{k}).$  On the other hand, since a graph $G$ of order $n$ with $ex_{con}(n,P_{k})+1$ must contains $P_{k}$, so $G$ must contain $F_m$. Hence the assertion holds. Moreover, it is obviously that the extremal graph is $Ex_{con}(n,P_{k})$.

(2) If there is exact one odd in  $\{k_1,k_2,\ldots,k_m\}$  with $k_m>3$, let $G_1=(\overline{K}_{n-\lfloor\frac{k}{2}\rfloor+1})+K_{\lfloor\frac{k}{2}\rfloor-1}$ and $G_2=K_{k+3}\bigcup\overline{K}_{n-k+2}+K_1$.  Since both of $G_1$ and $G_2$ contain no $F_m$  with $e(G_1)=[n,k]$ and $e(G_2)={k-2 \choose 2}+(n-k+2)$,  $ex_{con}(n,F_m)\ge \max\{{k-2 \choose 2}+(n-k+2),[n,k]\}$. On the other hand,  let $G$ be any graph with
$e(G)\ge \max\{{k-2 \choose 2}+(n-k+2),[n,k]\}+1$.
Then
\begin{eqnarray*}
e(G)&\geq&
 \max\left\{{k-2 \choose 2}+(n-k+2)+1,[n,k]+1\right\}\\
 &\geq& ex_{con}(n,P_k)\\
 &=& \max\left\{{k-2 \choose 2}+(n-k+2),[n,k]+1\right\},
\end{eqnarray*}
since there is exactly one odd in  $\{k_1,k_2,\ldots,k_m\}$.
If $e(G)> ex(n, P_k)$, then   $G$ contains $P_k$, i.e., $F_m$, by Theorem~\ref{conenctedpath}. If $e(G)=ex(n, P_k)$, then
 \begin{eqnarray*}
 e(G)&=& \max\left\{{k-2 \choose 2}+(n-k+2)+1,[n,k]+1\right\}\\
 &=&\max\left\{{k-2 \choose 2}+(n-k+2),[n,k]+1\right\}.
  \end{eqnarray*}
Hence  $e(G)=[n,k]+1$ and $n>\frac{5k-7}{4}$. By Theorem~\ref{conenctedpath},  either $G$ contains $P_k$, or $G=(K_{2}\bigcup\overline{K}_{n-\lfloor\frac{k+1}{2}\rfloor})
+K_{\lfloor\frac{k}{2}\rfloor-1}$, which contains $F_{m}$.  The assertion holds.

  (3) If $k_1,k_2,\ldots,k_{m-1}$ are all even and $k_m=3$,
   let $G$ be any graph with $e(G)\ge  ex_{con}(n,P_{k})$. If $e(G)>ex_{con}(n,P_{k})$, then $G$ contains $P_{k}$, i.e., $F_m$, by Theorem~\ref{conenctedpath}. If $e(G)=ex_{con}(n,P_{k})$ and $G$ contains $P_{k}$, then $G$ contains $F_{m}$, if $e(G)=ex_{con}(n,P_{k})$ and $G$ does not contain $P_{k}$, then  by Theorem~\ref{conenctedpath}, $G$ is either $(K_{k-3}\bigcup\overline{K}_{n-k+2})+K_{1}$ or $(K_{2}\bigcup\overline{K}_{n-\lfloor\frac{k+1}{2}\rfloor})+K_{\lfloor\frac{k}{2}\rfloor-1}$, both contain $F_m$. So the assertion holds.
    \end{proof}

\begin{remark}
(1) Let $k$ be an odd number. By a simple calculation, $[n,k]>{k-2 \choose 2}+(n-k+2)$ for $n>\frac{5k-7}{4}+\frac{2}{k-5}$;    $[n,k]<{k-2 \choose 2}+(n-k+2)$ for $n<\frac{5k-7}{4}+\frac{2}{k-5}$; and  $[n,k]={k-2 \choose 2}+(n-k+2)$ for $n=\frac{5k-7}{4}+\frac{2}{k-5}$.
(2) If there is  exact one odd in $\{k_1,k_2,\ldots,k_m\}$, it is interesting to determine $ex_{con}(n, F_m)$ and $Ex_{con}(n,F_m)$.
\end{remark}

\begin{lemma}\label{Pk1Pk2}
Let $G$ be a graph with $n$ vertices. If $k_1\geq k_2\geq 3,n\ge k_1+k_2$, $k_1,k_2$ are not both odd, then
$$ex(n,P_{k_1}\bigcup P_{k_2})= \max\left\{[n,k_{1},k_{1}],[n,k_{1}+k_{2},k_{2}],[n,k_{1}+k_{2}]\right\}.$$
Moreover, if $k_1,k_2$ are both even, then the extremal graphs are $$Ex(n,P_{k_1}), \ K_{k_1+k_2-1}\bigcup Ex(n-k_1-k_2+1,P_{k_2}),\  K_{\frac{k_1+k_2}{2}-1}+\overline{K}_{n-\frac{k_1+k_2}{2}+1}.$$
\end{lemma}
\begin{proof}
Let $G$ be any graph which does not contain $P_{k_1}\bigcup P_{k_2}$. Suppose that $e(G)> max\{[n,k_{1},k_{1}],[n,k_{1}+k_{2},k_{2}],[n,k_{1}+k_{2}]\}$. We consider the following two cases.

 {\bf Case 1.} $G$ is connected. Then  by Lemma~\ref{connected} and Theorem~\ref{conenctedpath},
\begin{eqnarray*}
e(G)&\leq& ex_{con}(n,P_{k_1}\bigcup P_{k_2})\\&\leq &\max\{{k_1+k_2-2 \choose 2}+(n-k_1-k_2+2),[n,k_{1}+k_{2}]\}\\&\leq & \max\{[n,k_{1}+k_{2},k_{2}],[n,k_{1}+k_{2}]\},
 \end{eqnarray*}
 the last inequality follows from Observation 1. This is a contradiction.

 {\bf Case 2.}  $G$ is disconnected. Since $e(G)>[n,k_1,k_1], $  $G$ must contain $P_{k_1}$ by theorem~\ref{Pathk2}.  Let $C$ be the component with $n_1\ge k_1$ vertices which contains $P_{k_1}$. Thus $C$ does not contain $P_{k_1}\bigcup P_{k_2}$ and $G-C$ does not contain $P_{k_2}$. Let $n_2=n-n_1$.  If $n_1\ge k_1+k_2$, then by Theorem~\ref{Pathk2} and Lemma~\ref{connected},
 \begin{eqnarray*}
  e(G)&\leq& ex_{con}(n_1,P_{k_1}\bigcup P_{k_2})+ex(n_2,P_{k_2})\\
  &\leq& \max\left\{ {k_1+k_2-2 \choose 2}+n_1-k_1-k_2+2,[n_1,k_1+k_2]\right\}+ [n_2,k_2,k_2]\\
  &< &\max\{[n,k_{1}+k_{2},k_{2}],[n,k_{1}+k_{2}]\}.
  \end{eqnarray*}
  The last inequality follows from the fact: (1) If ${k_1+k_2-2 \choose 2}+n_1-k_1-k_2+2\geq[n_1,k_1+k_2]$, then $n_1\leq\frac{5(k_1+k_2)-7}{4}+\frac{2}{k_1+k_2-5}$ for  $k_1+k_2$ being odd, and $n_1\leq\frac{5(k_1+k_2)-10}{4}$ for $k_1+k_2$ being even, which implies ${k_1+k_2-2 \choose 2}+n_1-k_1-k_2+2+[n_2,k_2,k_2]<{k_1+k_2-1 \choose 2}+[n_2,k_2,k_2]\leq [n,k_{1}+k_{2},k_{2}].$   (2) If ${k_1+k_2-2 \choose 2}+n_1-k_1-k_2+2\leq[n_1,k_1+k_2]$ then $[n_1,k_1+k_2]+[n_2,k_2,k_2]<[n,k_{1}+k_{2}]$ follows from observation 5. This is also a contradiction.  If $k_1\le n_1<k_1+k_2$, then
   \begin{eqnarray*}
  e(G)&\leq& ex_{con}(n_1,P_{k_1}\bigcup P_{k_2})+ex(n_2,P_{k_2})\\&\leq& {n_1 \choose 2}+ [n_2,k_2,k_2]\\
  &\le&\max\left\{[n,k_{1}+k_{2},k_{2}],[n,k_{1}+k_{2}]\right\},
  \end{eqnarray*}
   with the equality holds when $G=K_{k_1+k_2-1}\bigcup Ex(n-k_1-k_2+1,P_{k_2})$. This is a contradiction. Hence the assertion holds. Moreover, it's obviously that if $k_1,k_2$ are both even, then the extremal graphs are determined and we finish our proof.
  \end{proof}

\begin{remark}
If $k_1\leq 5 k_2$, it is easy to see that
 $$ex(n,P_{k_1}\bigcup P_{k_2})= \max\left\{[n,k_{1}+k_{2},k_{2}],[n,k_{1}+k_{2}]\right\} \mbox{ for } k_1\leq 5k_2.$$
\end{remark}

\begin{lemma}\label{km}
Let $n\ge\sum_{i=1}^{s}n_i,n_{i}\geq l_{i,1}+l_{i,2}+\ldots+l_{i,t_{i}}$, $l_{i,1}\geq l_{i,2}\geq \ldots \geq l_{i,t_{i}}\geq k_m,i\in\{1,2,\ldots,s\}$. If there is at most one odd in $\{l_{1,1},\ldots,l_{1,t_{1}},\ldots,\ldots,l_{s,1},\ldots,l_{s,t_{s}},k_m\}$,
then
 \begin{eqnarray*}
\sum_{i=1}^{s}ex_{con}(n_i,P_{l_{i,1}}\bigcup\ldots\bigcup P_{l_{i,t_{i}}}\bigcup P_{k_{m}})+ex(n-\sum_{i=1}^{s}n_i,P_{k_m})\\ \leq \max\left\{[n,\sum_{i=1}^{s}\sum_{j=1}^{t_i}l_{i,j}+k_{m},k_m],
[n,\sum_{i=1}^{s}\sum_{j=1}^{t_i}l_{i,j}+k_{m}]\right\}.
\end{eqnarray*}
\end{lemma}
\begin{proof}
By Lemma~\ref{connected} and Observation 1,
\begin{eqnarray*}
&&ex_{con}(n_i,P_{l_{i,1}}\bigcup\ldots\bigcup P_{l_{i,t_{i}}}\bigcup P_{k_{m}})\\& \le&
\max\left\{[n_i,\sum_{j=1}^{t_i}l_{i,j}+k_{m},k_m],[n_i,\sum_{j=1}^{t_i}l_{i,j}+k_{m}]\right\}.
\end{eqnarray*}
By Observations 3, 4, 6 and 7, we have
\begin{eqnarray*}
&&\max\left\{[n_1,\sum_{j=1}^{t_1}l_{1,j}+k_{m},k_m],[n_1,\sum_{j=1}^{t_1}l_{1,j}+k_{m}]\right\}\\&&+
\max\left\{[n_2,\sum_{j=1}^{t_2}l_{2,j}+k_{m},k_m],[n_2,\sum_{j=1}^{t_2}l_{2,j}+k_{m}]\right\}\\
&< & \max\left\{[n_1+n_2, \sum_{i=1}^2\sum_{j=1}^{t_i}l_{i,j}+k_{m}, k_m], [n_1+n_2,\sum_{i=1}^2\sum_{j=1}^{t_i}l_{i,j}+k_{m}]\right\}.
\end{eqnarray*}
Hence by Theorem~\ref{Pathk2}, we have
 \begin{eqnarray*}
&&\sum_{i=1}^{s}ex_{con}(n_i,P_{l_{i,1}}\bigcup\ldots\bigcup P_{l_{i,t_{i}}}\bigcup P_{k_{m}})+ex(n-\sum_{i=1}^{s}n_i,P_{k_m})\\
&\leq&\sum_{i=1}^{s} \max\left\{[n_i,\sum_{j=1}^{t_i}l_{i,j}+k_{m},k_m],[n_i,\sum_{j=1}^{t_i}l_{i,j}+k_{m}]\right\}+[n-\sum_{i=1}^{s}n_i,k_m,k_m]\\
&\leq & \max  \left\{
[\sum_{i=1}^sn_i,\sum_{i=1}^{s}\sum_{j=1}^{t_i}l_{i,j}+k_{m},k_m],
[\sum_{i=1}^sn_i,\sum_{i=1}^{s}\sum_{j=1}^{t_i}l_{i,j}+k_{m}]\right\}\\&&+[n-\sum_{i=1}^{s}n_i,k_m,k_m]\\
&\le & \max\left\{[n,\sum_{i=1}^{s}\sum_{j=1}^{t_i}l_{i,j}+k_{m},k_m],
[n,\sum_{i=1}^{s}\sum_{j=1}^{t_i}l_{i,j}+k_{m}]\right\},
\end{eqnarray*}
where the last inequality follows from Observations 2 and 5. Moreover, with equality holds if and only if $s=1$ and the equality occurs in Observation 2.
\end{proof}

\section{Proof of Theorem~\ref{main} and Corollaries}
Now we are ready to prove Theorem~\ref{main}.
\begin{proof}[Proof of theorem~\ref{main}]We  prove Theorem~\ref{main} by induction on $m$. For $m=2$, by Lemma~\ref{Pk1Pk2}, the assertion holds. Suppose it holds for smaller $m$. Let $G$ be any graph which does not contain $F_m$  and
$$e(G)> max\left\{[n,k_1,k_1],[n,k_1+k_2,k_2],\ldots,[n,\sum_{i=1}^{m}k_{i},k_{m}],
[n,\sum_{i=1}^{m}k_{i}]\right\}.$$
By  the induction hypothesis, $G$ must contain $F_{m-1}$. If $G$ is connected, by Lemma~\ref{connected},
\begin{eqnarray*}
e(G)&\leq& \max\left\{{\sum_{i=1}^{m}k_{i}-2 \choose 2}+n-\sum_{i=1}^{m}k_{i}+2,[n,\sum_{i=1}^{m}k_{i}]\right\}\\
&\leq& \max\left\{[n,\sum_{i=1}^{m}k_{i},k_m],[n,\sum_{i=1}^{m}k_{i}]\right\}\\
&\leq& \max\left\{[n,k_1,k_1],[n,k_1+k_2,k_2],\ldots,[n,\sum_{i=1}^{m}k_{i},k_{m}],[n,\sum_{i=1}^{m}k_{i}]\right\}
\end{eqnarray*}
 which is a contradiction. Suppose $G$ is disconnected. Since $G$ contains $F_{m-1}$, let $C_{i}$ be the component which contains $P_{l_{i,1}}\bigcup\ldots\bigcup P_{l_{i,t_i}}$, where $\{l_{i,1},\ldots, l_{i,t_i}\}\subseteq\{k_1,\ldots,k_{m-1}\}$, then $C_{i}$ does not contain $P_{l_{i,1}}\bigcup\ldots\bigcup P_{l_{i,t_i}}\bigcup P_{k_m}$  for $i=1,2,\ldots,s$. Further $$G-C_{1}\bigcup C_{2}\bigcup \ldots \bigcup C_{s} \mbox{ does not contain }P_{k_m}.$$ Let $v(C_{i})=n_i\ge \sum_{j=1}^{t_i}l_{i,j}$. By Lemma~\ref{km},
 \begin{eqnarray*}
 e(G)&\leq &\sum_{i=1}^{s}ex_{con}(n_i,P_{i_{1}}\bigcup\ldots\bigcup P_{i_{t}}\bigcup P_{k_m})+ex(n-\sum_{i=1}^{s}n_i,P_{k_m})\\&\leq & \max\left\{[n,\sum_{i=1}^{m}k_{i},k_{m}],[n,\sum_{i=1}^{m}k_{i}]\right\},
 \end{eqnarray*}
 which is a contradiction.

Let all of $k_1,k_2,\ldots,k_m$ be even. If $G$ is disconnected and the equality occurs in lemma~\ref{km}, then the equality must occur in observations 2. Moreover, if $G$ is connected, by lemma~\ref{connected}, the extremal graphs are determined. Hence, by induction, it is easy to see that the extremal graphs are characterized, which are
$$Ex(n,P_{k_1}),\ldots,Ex(n-\sum_{i=1}^{m}k_i+1,P_{k_m})\bigcup K_{\sum_{i=1}^{m}k_i-1},$$
$$K_{\sum_{i=1}^{s}\frac{k_i}{2}-1}+\overline{K}_{n-\sum_{i=1}^s\frac{k_i}{2}+1}.$$
In addition, $Ex(n, P_{k_1}), \ldots, Ex( n-\sum_{i=1}^mk_i+1, P_{k_m})$ are described in Theorem~\ref{Pathk2}.
 \end{proof}

In particular,
\begin{corollary}\cite{Bushwa2011}
Let $G$ be a graph with $n$ vertices. If $l$ is an even number, then
$$ex(n,k\cdot P_{l})= \max\left\{[n,kl,l],[n,kl]\right\}.$$
Moreover the extremal graphs are
$$Ex(n-kl+1,P_l)\bigcup K_{kl-1},K_{\frac{kl}{2}-1}+\overline{K}_{n-\frac{kl}{2}+1}.$$
\end{corollary}
\begin{proof}
Since $[n,kl,l]>[n,(k-1)l,l]>\ldots>[n,2l,l]>[n,l,l]$, the assertion follows from  Theorem~\ref{main}.
\end{proof}

\remark In \cite{Bushwa2011}, Bushaw and Kettle showed that the graph $$K_{k\lfloor\frac{l}{2}\rfloor-1}+(K_i+\overline{K}_{n-k\lfloor\frac{l}{2}\rfloor-i+1})$$ is the extremal graph of $ex(n,k\cdot P_l)$ for $k\ge 2,l\ge 4,$ and $n\ge2l+2kl(\lceil\frac{l}{2}\rceil+1){l \choose \lfloor\frac{l}{2}\rfloor}$, ($i=1$, when $l$ is even, $i=2$ when $l$ is odd). Based on this result, they conjectured that this construction is optimal for $ n=\Omega(kl)$. Let $n=(kl-1)+t(l-1)+r,0\le r<l-1.$
An simple calculation shows that if $n\geq
\frac{5}{4}kl$, then $[n,kl]\geq[n,kl,l]$. In other words, we confirm their conjecture when $l$ is even.

\begin{corollary}
$ex(n,P_{6k}\bigcup P_{6k}\bigcup P_{4k})=ex(n,P_{8k}\bigcup P_{4k}\bigcup P_{4k})$ for all $n,k$. Moreover, if $n\ge 14k$, the extremal graphs are $K_{16k-1}\bigcup Ex(n-16k+1,4k)$ and $K_{8k-1}+\overline{K}_{n-8k+1}$.
\end{corollary}
\begin{proof}
If $n<14k$, Clearly the assertion holds. So we may assume $n\ge 14k$. By Theorem~\ref{main},
\begin{eqnarray*}
&&ex(n,P_{6k}\bigcup P_{6k}\bigcup P_{4k})\\&=& \max\left\{[n,16k,4k],[n,12k,6k],[n,6k,6k],[n,16k]\right\}\\&=& \max\left\{[n,16k,4k],[n,12k,6k],[n,16k]\right\}\\&=& \max\left\{[n,16k,4k],[n,16k]\right\}.
\end{eqnarray*}
The third quality follows form the following two cases: (1) If $n\geq18k-2$, then $[n,16k]\geq[n,12k,6k]$. (2) If $n\leq18k-2$, then $[n,16k,4k]\geq[n,12k,6k]$. Oh the other hand,
\begin{eqnarray*}
&&ex(n,P_{8k}\bigcup P_{4k}\bigcup P_{4k})\\&=& \max\left\{[n,16k,4k],[n,12k,4k],[n,8k,8k],[n,16k]\right\}\\&= &\max\left\{[n,16k,4k],[n,8k,8k],[n,16k]\right\}\\&= & \max\left\{[n,16k,4k],[n,16k]\right\}.
\end{eqnarray*}
Hence the assertion holds.
\end{proof}

\begin{remark}
 In \cite{erdHos1976},  Erd\H{o}s and Simonovits asked that if $F_1$ and $F_2$ are two bipartite graphs, Giving conditions on $F_1$ and $F_2$ ensuring that $ex(n,F_1)=ex(n,F_2)$, provided $n$ is sufficiently large (also see \cite{bollobas1979}, chapter 6, problem 41). Let $F_m=P_{k_1}\bigcup\ldots\bigcup P_{k_m},F_{m^{\prime}}^{\prime}=P_{k^{\prime}_1}\bigcup\ldots\bigcup P_{k^{\prime}_{m^{\prime}}}$ and all of $\{k_1,\ldots,k_m\}$ are odd if and only if all of $\{k^{\prime}_1,\ldots,k^{\prime}_{m^{\prime}}\}$ are odd. By Theorem~\ref{Fm} \cite{lidicky2013}, if  $\sum_{i=1}^{m}\lfloor\frac{k_i}{2}\rfloor=\sum_{i=1}^{m^{\prime}}\lfloor\frac{k^{\prime}_i}{2}\rfloor$,  then $ex(n,F_{m^{\prime}}^{\prime})=ex(n,F_m)$, provided $n$ is sufficiently large. Our results show that there exist two family graphs $F_m$ and $F_m^{\prime}$ such that $ex(n, F_m)=ex(n, F_m^{\prime})$ for all $n$.
 \end{remark}

\section{Proof of Theorem~\ref{p2l+1,3}}
In order to prove Theorem~\ref{p2l+1,3},  we need the following notations and several Lemmas. Let $G=(V, E)$ be a simple graph. If $u$ and $v$ in $V$ are adjacent, we say that $u$ {\it hits} $v$ or $v$ {\it hits} $u$. If $u$ and $v$ are not adjacent, we say that $u$ {\it misses} $v$ or $v$ {\it misses} $u$.
\begin{lemma}\label{5,3}
Let $G$ be a graph with $n$ vertices. If $n\ge 8$, then
$$ex(n,P_{5}\bigcup P_{3})= \max\left\{21+\lfloor \frac{n-7}{2}\rfloor,2(n-1)\right\}.$$
Moreover, the extremal graphs are $K_7\bigcup M_{n-7}$ and $K_2+(K_2\bigcup \overline{K}_{n-4})$.
\end{lemma}
\begin{proof}
Let $G$ be any graph which does not contain $P_{5}\bigcup P_{3}$ with $e(G)\geq max\{21+\lfloor \frac{n-7}{2}\rfloor,2(n-1)\}$. We consider the following two cases.

  {\bf Case 1.}  $G$ is connected. Suppose that $G\neq K_2+(K_2\bigcup \overline{K}_{n-4})$. By Theorem~\ref{conenctedpath}, $G$ contains a $P_{7}$. Let $P_{7}=x_{1}x_{2}\ldots x_{7}$ be a subgraph in $G$. First we will show that there is no edge in $G-P_{7}$.  Clearly, if there is an edge in $G-P_{7}$ then $G$ contains $P_{5}\bigcup P_{3}$, this is a contradiction.  If all the vertices in $G-P_{7}$ hit exact one vertex in $P_{7}$, then they must hit $x_{3}$ or $x_{5}$, say $y_{1}$ hits $x_{3}$. Obviously, $\{x_{1},x_{2}\}$ can't hit $ \{x_{4},x_{5},x_{7}\}$ and if $x_{1}$ or $x_{2}$ hits $x_{6}$, then $x_{7}$ must miss $x_{5}$. Hence
$$e(G)\leq{7 \choose 2}-6-1+n-7=n+7<\max\left\{21+\lfloor \frac{n-7}{2}\rfloor,2(n-1)\right\}.$$
If at least one of the vertices in $G-P_{7}$ hits two vertices in $P_{7}$, then there is at most one edge $x_{2}x_{6}$ among $\{x_{1},x_{2}\},x_{4},\{x_{6},x_{7}\}$ and $x_{3}$ can't hit $\{x_{6},x_{7}\},$ $x_{5}$ can't hit $\{x_{1},x_{2}\}$. Hence
$$e(G)\leq 2(n-7)+{7 \choose 2}-\left({5 \choose 2}-3\right)-4=2n-4<\max\left\{21+\lfloor \frac{n-7}{2}\rfloor,2(n-1)\right\}.$$
Both are  contradictions.

 {\bf Case 2.} $G$ is disconnected.  By Theorem\ref{Pathk2}, $G$ contains $P_5$, Let $C$ be the component with $n_1\ge5$ vertices which contains a $P_5$. Let $n_2=n-n_1$. If $n_1\ge 8$, then by the similar argument,
$$e(C)\leq  \max\left\{21+\lfloor \frac{n_1-7}{2}\rfloor,2(n_1-1)\right\},e(G-C)\leq \lfloor\frac{n_2}{2}\rfloor.$$
Hence
\begin{eqnarray*}
e(G)&\leq &e(C)+e(G-C)\\ &\leq& \max\left\{21+\lfloor \frac{n_1-7}{2}\rfloor,2(n_1-1)\right\}+ \lfloor\frac{n_2}{2}\rfloor\\&< &\max\left\{21+\lfloor \frac{n-7}{2}\rfloor,2(n-1)\right\},
\end{eqnarray*}
where  the second inequality becomes equality if and only if $G-C=K_2+(K_2\bigcup \overline{K}_{n_1-4})$.
If $n_1\leq7$, then
\begin{eqnarray*}
e(G)&\leq &e(G-C)+e(C)\\ &\leq& {n_1 \choose 2}+ \lfloor\frac{n_2}{2}\rfloor\\&\leq &\max\left\{21+\lfloor \frac{n-7}{2}\rfloor,2(n-1)\right\},
\end{eqnarray*}
with equality when $G=K_7\bigcup M_{n-7}$. So the assertion holds.
\end{proof}

\begin{lemma}\label{circle}
Let $G\neq K_{2l+3}\bigcup M_{n-2l-3}$ be a graph with $n$ vertices and $$e(G)\geq{2l+3 \choose 2}+\lfloor \frac{n-2l-3}{2}\rfloor.$$ If $G$ contains  either $C_{2l+2}$ or $C_{2l+3}$, then $G$ contains $P_{2l+1}\bigcup P_3.$
\end{lemma}
\begin{proof}
Suppose $G$ contains no $P_{2l+1}\bigcup P_3.$ If $G $ contains $C_{2l+3}$, then any vertex in $G-C_{2l+3}$ can't hit the vertices in $C_{2l+3}$. Hence $e(G)<{2l+3 \choose 2}+\lfloor\frac{n-2l-3}{2}\rfloor$, which is a contradiction. If $G$ contains $C_{2l+2}$, then each component of $G-C_{2l+2}$ is either isolated vertex or edge. It is easy to see that the vertices of the edge in $G-C_{2l+2}$ can not hit $C_{2l+2}$, and any two of the isolated vertices in $G-C_{2l+2}$ can not hit the same vertex of $C_{2l+2}$ or any two consecutive vertices of $C_{2l+2}$. Therefore $e(G)\leq{2l+2 \choose 2}+l+1+\lfloor \frac{n-3l-3}{2}\rfloor<{2l+3 \choose 2}+\lfloor \frac{n-2l-3}{2}\rfloor$, which is also a contradiction.
\end{proof}
Now we are ready to prove Theorem~\ref{p2l+1,3}.
\begin{proof}[Proof of Theorem~\ref{p2l+1,3}] For $l=2$, the assertion follows from Lemma~\ref{5,3}. Hence we may assume $l\geq 3$. Let $G$ be any graph which does not contain $P_{2l+1}\bigcup P_{3}$ and
$$e(G)\geq \max\left\{{2l+3 \choose 2}+\lfloor \frac{n-2l-3}{2}\rfloor,{l \choose 2}+l(n-l)+1,[n,2l+1,2l+1]\right\}.$$
Then by Theorem~\ref{Pathk2}, $G$ contains $P_{2l+1}$.

{\bf Case 1.} $G$ does not contain $P_{2l+3}$, we claim that $G$ is connected. In fact, if $G$ is disconnected,
then one of the components, says $C$ with $n_1$ vertices, must contain $P_{2l+1}$ and the other component is edge or isolated vertex. Hence
\begin{eqnarray*}
e(G)&=&e(C)+e(G-C)\\&\leq &ex_{con}(n_1,P_{2l+3})+\lfloor\frac{n-n_1}{2}\rfloor\\&\leq &\max\left\{{2l+1 \choose 2}+n_1-2l-1,{l \choose 2}+l(n_1-l)+1\right\}+\lfloor\frac{n-n_1}{2}\rfloor\\&<&
\max\left\{{2l+3 \choose 2}+\lfloor \frac{n-2l-3}{2}\rfloor,{l \choose 2}+l(n-l)+1,[n,2l+1,2l+1]\right\},
\end{eqnarray*}
which also is a contradiction. Further  by Theorem~\ref{conenctedpath},
\begin{eqnarray*}
e(G)&\leq& \max\left\{{2l+1 \choose 2}+n-2l-1,{l \choose 2}+l(n-l)+1\right\}\\
&\leq&\max\left\{{2l+3 \choose 2}+\lfloor \frac{n-2l-3}{2}\rfloor,{l \choose 2}+l(n-l)+1,[n,2l+1,2l+1]\right\},
\end{eqnarray*}
with the quality holds when $G=K_l+(K_2\bigcup \overline{K}_{n-l-2})$, where the last inequality follows from  ${l \choose 2}+l(n-l)+1\geq{2l+1 \choose 2}+n-2l-1$ for  $n\geq\frac{5l-1}{2}$, and  ${2l+1 \choose 2}+n-2l-1<{2l+3 \choose 2}+\lfloor \frac{n-2l-3}{2}\rfloor$ for $n<\frac{5l-1}{2}$.  So the assertion holds.

{\bf Case 2.} $G$ contains $P_{2l+3}$. Let $P_{2l+3}=x_1x_2\ldots x_{2l+3}$, $Y=G-P_{2l+3}$ and $V(Y)=\{y_1,y_2,\ldots,y_{n-2l-3}\}$, $d_{P_{2l+3}}(y_i)$ be the number of vertices which adjacent to $y_i$ in $P_{2l+3}$ for $i=1,2,\ldots,n-2l-3$. Obviously, $y_{i}$ can not hit $x_{1},x_{2},x_{2l+2},x_{2l+3}$, moreover $y_{i}$ can not hit both vertices of $\{x_{k},x_{k+1}\}$ or $\{x_{k},x_{k+4}\}$  for $k=1,2,\ldots,2l+3$. So $d_{P_{2l+3}}(y_{i})\leq l-1$. Let $y$ be a vertex in $Y$ with $d_{P_{2l+3}}(y)$ being maximum value,  and $x_{i_1},x_{i_2},\ldots,x_{i_s}$ be the all neighbours of $y$ in $P_{2l+3}$, if $s=0$, then $G[P_{2l+3}]$ is a component of $G$, the result follows. Hence we may assume $s\ge1$.

{\bf Claim.} There are $2s$ distinct vertices in $P_{2l+3}$ which form $s$ pairs vertices whose degree sum is at most $2l+3$.

{\bf Fact 1.} $i_{k+1}-i_{k}\neq 4$. Because $G$  does not contain $P_{2l+3}\bigcup P_3$.

{\bf Fact 2.} $d_{P_{2l+3}}(x_{i_k-1})+d_{P_{2l+3}}(x_{i_k+2})\leq2l+3$. Let $x_{p}$ be a neighbor of $x_{i_k-1}$. If $p<i_k-1$, then $x_{i_k+2}$ can not hit $x_{p+1}$, otherwise, $x_1x_2\ldots x_px_{i_k-1}\ldots x_{p+1}x_{i_k+2}\ldots x_{2l+3}$ together with $yx_{i_k}x_{i_k+1}$ is a $P_{2l+1}\bigcup P_3$ in $G$, a contradiction. Similarly, if $p>i_k+1$, $x_{i_k+2}$ can not hit $x_{p+1}$. Let $d_{P_{2l+3}}(x_{i_k-1})=z$. Since $x_{i_k+2}$ can not hit $x_1$, we have  $d_{P_{2l+3}}(x_{i_k-1})+d_{P_{2l+3}}(x_{i_k+2})\leq z+2+2l+2-z-1=2l+3$.

{\bf Fact 3.} $d_{P_{2l+3}}(x_{i_k-2})+d_{P_{2l+3}}(x_{i_k+1})\leq2l+3$.
The proof of this fact is similar to the proof of Fact 2.

Let $x_{i_{j_l}}$ be the neighbor of $y$ such that $x_{i_{j_l}-2}$ is also the neighbor of $y$  for $l=1,2,\ldots,t$. Obviously, $\{x_{i_{j_1}},x_{i_{j_2}},\ldots,x_{i_{j_t}}\}$ divides $P_{2l+3}$ into $t+1$ parts. By Fact 1, we can choose pairs of vertices in each part by $\{x_{i_k-1},x_{i_k+2}\},\{x_{i_{k+1}-2},x_{i_{k+1}+1}\}$ alternately. In the first part, we choose
$\{x_{{i_1}-1},x_{{i_1}+2}\},$ $\{x_{{i_2}-2},x_{{i_2}+1}\},$ $\{x_{{i_3}-1},x_{{i_3}+2}\},$ $\ldots,$ $\{x_{{i_{j_1}}-4},x_{{i_{j_1}}-1}\}$ or $ \{x_{i_{j_1}-3},x_{i_{j_1}}\}.$ In the following parts, we will always begin with the pair $\{x_{i_{j_l}-2},x_{i_{j_l}+1}\},$ for $l=1,2,\ldots,t$. So in the second part, we choose
$\{x_{i_{j_1}-2},x_{i_{j_1}+1}\},$ $\{x_{i_{j_1+1}-1},x_{i_{j_1+1}+2}\},$ $\{x_{i_{j_1+2}-2},x_{i_{j_1+2}+1}\},$ $\ldots,$ $\{x_{i_{j_2}-4},x_{{i_{j_2}}-1}\}$ or $ \{x_{i_{j_2}-3},x_{i_{j_2}}\}.$ The process will go on until in the last part we choose
$\{x_{i_{j_t}-2},x_{i_{j_t}+1}\},\{x_{i_{j_t+1}-1},x_{i_{j_t+1}+2}\},\{x_{i_{j_t+2}-2},x_{i_{j_t+2}+1}\},\ldots,\\ \{x_{{i_s}-2},x_{{i_s}+1}\}$ or $\{x_{{i_s}-1},x_{{i_s}+2}\}.$ By Facts 2 and  3, those $s$ pairs vertices whose degree sum is at most $2l+3$. Thus we finish our claim.

 Those $s$ pairs of vertices together with $\{x_1,x_{2l+2}\}$ or $\{x_1,x_{2l+3}\}$ are distinct vertices, and $d_{P_{2l+3}}(x_1)+d_{P_{2l+3}}(x_{2l+2})\leq 2l+1$, $d_{P_{2l+3}}(x_1)+d_{P_{2l+3}}(x_{2l+3})\leq 2l+1$. In fact, if $d_{P_{2l+3}}(x_1)+d_{P_{k+2}}(x_{2l+2})\geq 2l+2$ or $d_{P_{2l+3}}(x_1)+d_{P_{2l+3}}(x_{2l+3})\geq 2l+2$, $G$ must contain $C_{2l+2}$ or $C_{2l+3}$, by Lemma~\ref{circle}, $G$ contains $P_{2l+3}\bigcup P_3$, a contradiction. Hence, we have
$$e(G)\leq {2l+3 \choose 2}-\lceil\frac{2s(l+1)+2l+3}{2}\rceil+s(n-2l-3)+\lfloor\frac{n-2l-3}{2}\rfloor.$$
We will consider two cases.
(1) If $n\leq 3l+5$, then $e(G)<{2l+3 \choose 2}+\lfloor \frac{n-2l-3}{2}\rfloor$. (2) If $n\geq 3l+6$, we will show that $e(G)<{l \choose 2}+l(n-l)+1$. Since $${l \choose 2}+l(n-l)+1-\left[{2l+3 \choose 2}-\lceil\frac{2s(l+1)+2l+3}{2}\rceil+s(n-2l-3)+\lfloor\frac{n-2l-3}{2}\rfloor\right]$$ is increasing with respect to $n$, we only to check $n=3l+6$, that is ${2l+3 \choose 2}-[s(l+1)+l+2]+s(l+3)+\lfloor\frac{l+3}{2}\rfloor<{l \choose 2}+(2l+6)l+1$, this is true for $l\geq 3$. By (1) and (2),
\begin{eqnarray*}
e(G)&<& \max\left\{{2l+1 \choose 2}+n-2l-1,{l \choose 2}+l(n-l)+1\right\}\\
&<&\max\left\{{2l+3 \choose 2}+\lfloor \frac{n-2l-3}{2}\rfloor,{l \choose 2}+l(n-l)+1,[n,2l+1,2l+1]\right\},
\end{eqnarray*}
which is a contradiction. The proof is completed.
\end{proof}

\section{Proof of Theorem~\ref{p5,p5}}
In order to prove Theorem~\ref{p5,p5},  we need the following Lemma.
\begin{lemma}\label{conected,p5,p5}Let $G$ be a connected graph with $n$ vertices. If $n\ge 10$, then
 $$ex_{con}(n,P_{5}\bigcup P_{5})\leq \max\{[n,10,5],3n-5\}.$$
 Moreover if $ex_{con}(n,P_{5}\bigcup P_{5})=3n-5$, then $G=(K_2\bigcup\overline{K}_{n-5})+K_3$.
\end{lemma}
\begin{proof} Let $G\neq (K_2\bigcup\overline{K}_{n-5})+K_3$ be any connected graph which does not contain $P_5\bigcup P_5$  with
$e(G)\geq  \max\{[n,10,5],3n-5\}$. Then $ \max\{[n,10,5],3n-5\}\ge ex_{con}(n,P_9)$.  By Theorem~\ref{conenctedpath}, $G$ contains $P_9$. Let $P_9= x_1x_2\ldots x_9$ be a subgraph of $G$. Then  each vertex in $G-P_9$ misses $\{x_1,x_4,x_6,x_9\}$ and can not hit both vertices of $\{x_2,x_8\}$. Moreover, if $y$ is not an isolated vertex in $G-P_9$, then $y$ can only hit $x_5$,  otherwise $G$ contains $P_5\bigcup P_5$. First, we will prove the following Facts.

{\bf Fact 1.} If an edge of $G-P_9$ hits $P_9$, then $e(G[P_9])\leq24.$

Let $y_1y_2$ be an edge in $G-P_9$, $y_1$ hits $x_5$. Then $\{x_1,x_2\}$ misses $\{x_6,x_7,x_8,x_9\}$ and $\{x_3,x_4\}$ misses $\{x_8,x_9\}$. So $e(G[P_9])\leq36-12=24$.

{\bf Fact 2.}  If a $P_3=y_1y_2y_3$ of $G-P_9$ such that $y_1$ hits $P_9$, then $e(G[P_9])\leq21.$

 Clearly, $y_1$ must hit $x_5$, $\{x_1,x_2,x_3\}$ misses $\{x_6,x_7,x_8,x_9\}$ and $ x_4 $ misses $\{x_7,x_8,x_9\}$. So $e(G[P_9])\leq36-15=21$.

{\bf Fact 3.} If two isolated vertices both hit three vertices of $P_9$, then they must hit the same vertices. Moreover $e(G[P_9])\leq21.$

Let $y_1,y_2$ be two vertices both hit three vertices of $P_9$. If $y_1$ hits $x_2,x_5,x_7$, then $y_2$ can not hit $x_3$, otherwise $y_2$ hits $x_3$ which implies that $x_4x_3y_2x_5x_6$, $x_1x_2y_1x_7x_8$ are two disjoint $P_5$. Moreover, $y_2 $ can't hit $x_8$, otherwise $x_1x_2y_2x_8x_9$, $x_3x_4x_5x_6x_7$ are two disjoint $P_5$. Hence $y_2$ hits $x_2,x_5,x_7$. Further it is  easy to see that there is no edge among $x_1,\{x_3,x_4\},x_6,\{x_8,x_9\}$ and $\{x_3,x_4\}$ misses $x_7$. Then $e(G[P_9])\leq36-15=21$.
If $y_1$ hits $x_3,x_5,x_7$, then $y_2$ can not hit $x_2,x_8$, otherwise  $y_2$ hits $x_2$ which implies that $x_1x_2y_2x_7x_8$ $(x_1x_2y_2x_7x_8)$ and $x_4x_3y_1x_5x_6$ are two disjoint $P_5$. Hence $y_2$ hits $x_3,x_5,x_7$.  It is  easy to see that there is no edge among $\{x_1,x_2\},x_4,x_6,\{x_8,x_9\}$ and $\{x_1,x_2\}$ misses $x_7$. Then $e(G[P_9])\leq36-15=21$.

{\bf Fact 4.} If an isolated vertex hit two vertices of $P_9$, then $e(G[P_9])\leq29.$

Let $y$ be an isolated vertex in $G-P_9$ which hits exact two vertices of $P_9$. If $y$ hits $\{x_2,x_5\}$, then $\{x_3,x_4\}$ misses $\{x_7,x_9\}$ and  $x_1$ misses $\{x_4,x_6,x_9\}$. If $y$ hits $\{x_3,x_5\}$, then $\{x_1,x_2\}$ misses $\{x_7,x_9\}$, $x_1$ misses $\{x_4,x_6\}$, and $x_9$ misses $x_4$. If $y$ hits $\{x_2,x_7\}$, then $\{x_3,x_5\}$ misses $\{x_8,x_9\}$, and $x_1$ misses $\{x_3,x_4,x_6\}$. If $y$ hits $\{x_3,x_7\}$, then $\{x_4,x_6\}$ misses $\{x_1,x_2,x_8,x_9\}$. In any situation, it is easy to see  that $e(G[P_9])\leq36-7=29.$

{\bf Fact 5.} If an isolated vertex hits one vertex of $P_9$, then $e(G[P_9])\leq33.$

Let $y$ be an isolated vertex in $G-P_9$ which hits only one vertex of $P_9$. If $y$ hits $x_2$, then $x_1$ misses $\{x_4,x_6,x_9\}$. If $y$ hits $x_3$, then $\{x_1,x_2\}$ misses $\{x_7,x_9\}$. If $y$ hits $x_5$, then $x_1$ misses $\{x_6,x_9\}$ and $x_9$ misses $\{x_1,x_4\}$. In any situation, it is easy to see that $e(G[P_9])\leq36-3=33.$

Now we consider the following two cases.

{\bf Case 1.}  There is an edge in $G-P_9$.
Let $P_k$ be a longest path start at $x_5$ in $G[x_5\bigcup V(G-P_9)]$. If $k\geq 4, $ then  the number of edges incident with the vertices of $G-P_9$ is at most $3(n-9)$.
Since $G-P_9$ can't contain $P_5$, $y$ can only hit $x_5$ for $y$ being not an isolated vertex in $G-P_9$ and an isolated vertex in $G-P_9$ hits at most three vertices of $P_9$.
By Fact 2,  we have $e(G)\leq21+3(n-9)<\max\{[n,10,5],3n-5\}$, a contradiction.  If  $k\leq 3$,  each component of $G-P_9$  is a star (with at least three vertices), or an edge,  or an isolated vertex. Clearly, only the center of the star (the vertex of the star with degree more than one) can hit $x_5$. Hence the number of edges incident with the vertices of $G-P_9$ is at most $3(n-9)-3$. So by Fact 1, $e(G)\leq 24+3(n-9)-3<\max\{[n,10,5],3n-5\}$, which is also a contradiction.

{\bf Case 2:} There are no edges in $G-P_9$. If  there are at least two vertices which hits three vertices of $P_9$, then by Fact 3,  we have $e(G)\leq21+ 3(n-9)<\max\{[n,10,5],3n-5\}$, a contradiction. If all vertices of $G-P_9$ hit only one vertex of $P_9$, then by Fact 5,  we have $e(G)\leq33+(n-9)<\max\{[n,10,5],3n-5\}$, a contradiction. If  there is at least one vertex which hits two vertices of $P_9$ and there is at most one vertex which hits three vertices of $P_9$, then by Fact 4, we have $e(G)\leq29+2(n-9)+1<\max\{[n,10,5],3n-5\}$, a contradiction. So the assertion holds.
\end{proof}
Now we are ready to prove Theorem~\ref{p5,p5}.
\begin{proof}[Proof of Theorem~\ref{p5,p5}]
Let $G$ be any graph which does not contain $P_5\bigcup P_5$ with $e(G)\geq max\{[n,10,5],3n-5\}$. If $G$ is connected, then the assertion  follows from Lemma~\ref{conected,p5,p5}. If $G$ is disconnected, $ G$ contains $P_5$ by  $e(G)>ex(n,P_5)$.  Let $C$ be a component with $n_1\ge 5$ vertices which contains $P_5$. Obviously $C$ contains no $P_5\bigcup P_5$ and $G-C$ contains no $P_5.$ If $n_1\ge 10,$ then
$$e(G)\leq \max\{[n_1,10,5],3n_1-5\}+[n-n_1,5,5]< \max\{[n,10,5],3n-5\}.$$
 If $n_1\leq9,$ then $e(G)\leq {n_1 \choose 2}+[n-n_1,5,5]\leq \max\{[n,10,5],3n-5\}$ with the equality holds only when $n_1=9$ and $G=K_9\bigcup Ex(n,P_5)$. The proof is completed.
\end{proof}

\section{Conclusion}
Theorems~\ref{main}, ~\ref{p2l+1,3} and \ref{p5,p5} show that $$ex(n,F_m)=\max\left\{[n,k_1,k_1],[n,k_1+k_2,k_2],\ldots,[n,\sum_{i=1}^{m}k_{i},k_{m}]\right\} \mbox{ for small }n,$$ while Theorem~\ref{Fm} determines the value $ex(n,F_m)$  for $n$ sufficiently large.  So we may propose the following conjecture.
\begin{conjecture}
Let $k_1\geq k_2\geq\ldots\geq k_{m}\geq 3$ and $k_1>3$. If $F_m=
P_{k_1}\bigcup P_{k_2}\bigcup \ldots \bigcup P_{k_m}$, then
\begin{eqnarray*}
e(n, F_m)&= &\max\left\{\left[n,k_1,k_1\right],\left[n,k_1+k_2,k_2\right],\ldots,\left[n,\sum_{i=1}^{m}k_{i},k_{m}\right],\left[n,\sum_{i=1}^{m}\lfloor \frac{k_{i}}{2}\rfloor\right]+c\right\},
\end{eqnarray*}
where $c=1$ if all of $k_1,k_2,\ldots,k_m$ are odd, and $c=0$ for otherwise. Moreover, the extremal graphs are
$$Ex(n,P_{k_1}),\ldots,K_{\sum_{i=1}^{m}k_i-1}\bigcup Ex(n-\sum_{i=1}^{m}k_i+1,P_{k_m}), \mbox{and} $$
$$\begin{array}{ll}K_{\sum_{i=1}^{m}\lfloor \frac{k_{i}}{2}\rfloor-1}+(K_2\bigcup\overline{K}_{n-\sum_{i=1}^{m}\lfloor \frac{k_{i}}{2}\rfloor-1})& \mbox{ if all of }\{k_1,k_2,\ldots,k_m\}\mbox{ are odd},\\
K_{\sum_{i=1}^{m}\lfloor \frac{k_{i}}{2}\rfloor-1}+(\overline{K}_{n-\sum_{i=1}^{m}\lfloor \frac{k_{i}}{2}\rfloor+1})& \mbox{  otherwise}.
\end{array}$$
\end{conjecture}

\newpage
\appendix
\section{Proof of observations}
{\bf Observation 1:} Let $n\ge k_1+k_m$. Then
\begin{eqnarray*}
&&\max\left\{{k_1+k_m-2 \choose 2}+n-k_1-k_m+2,[n,k_1+k_m]\right\}\\&\leq &\max\left\{[n,k_1+k_m,k_m],[n,k_1+k_m]\right\}.
\end{eqnarray*}
{\it Proof.} We consider the following two cases.

 {\bf Case 1}. $k_1+k_m$ is odd. If $n>\frac{5(k_1+k_m)-7}{4}+\frac{2}{k_1+k_m-5}$, then
$${k_1+k_m-2 \choose 2}+n-k_1-k_m+2<[n,k_1+k_m].$$
If $n\le\frac{5(k_1+k_m)-7}{4}+\frac{2}{k_1+k_m-5}$, then
$${k_1+k_m-2 \choose 2}+n-k_1-k_m+2<{k_1+k_m-1 \choose 2}\leq[n,k_{1}+k_{m},k_{m}].$$
{\bf Case 2}. $k_1+k_m$  is even. If $n>\frac{5(k_1+k_m)-10}{4}$, then
$${k_1+k_m-2 \choose 2}+n-k_1-k_m+2<[n,k_1+k_m].$$
If $n\le\frac{5(k_1+k_m)-10}{4}$, then
$${k_1+k_m-2 \choose 2}+n-k_1-k_m+2<{k_1+k_m-1 \choose 2}\leq[n,k_{1}+k_{m},k_{m}].$$
Hence
 the assertion holds.

{\bf Observation 2:} Let $n_1\ge k_1$. Then $$[n_1,k_1+k_m,k_m]+[n_2,k_m,k_m]\leq[n_1+n_2,k_1+k_m,k_m].$$

{\it Proof.} Let $n_1=k_1+t_1(k_m-1)+r_1,n_2=t_2(k_m-1)+r_2$ and $n_1+n_2=k_1+t_3(k_m-1)+r_3$, where $0\le r_1, r_2, r_3< k_m-1.$ If $t_1\ge 1$, then
\begin{eqnarray*}
&&[n_1,k_1+k_m,k_m]+[n_2,k_m,k_m]\\&=&{k_1+k_m-1 \choose 2}+(t_1-1){k_m-1 \choose 2}+{r_1 \choose 2}+t_2{k_m-1 \choose 2}+{r_2 \choose 2}\\&\leq&{k_1+k_m-1 \choose 2}+(t_3-1){k_m-1 \choose 2}+{r_3 \choose 2}\\&=&[n_1+n_2,k_1+k_m,k_m],
\end{eqnarray*}
with equality only when $r_1=0$ or $r_2=0$. If $t_1=0$, it is easy to see that the observation holds, moreover the equality can not occur.

{\bf Observation 3:} Let $n_1 \ge k_1$ and $n_2\ge k_2$. Then $$[n_1,k_1+k_m,k_m]+[n_2,k_2+k_m,k_m]<[n_1+n_2,k_1+k_2+k_m,k_m].$$

{\it Proof.} Let $n_1=k_1+t_1(k_m-1)+r_1,n_2=k_2+t_2(k_m-1)+r_2,n_1+n_2=k_1+k_2+t_3(k_m-1)+r_3$, where $0\le r_1, r_2, r_3< k_m-1.$  If $k_1\ge 1$ and $k_2\ge 1$, then
\begin{eqnarray*}
&&
[n_1,k_1+k_m,k_m]+[n_2,k_2+k_m,k_m]\\&=&{k_1+k_m-1 \choose 2}+(t_1-1){k_m-1 \choose 2}+{r_1 \choose 2}\\&&+{k_2+k_m-1 \choose 2}+(t_2-1){k_m-1 \choose 2}+{r_2 \choose 2}\\&<&{k_1+k_2+k_m-1 \choose 2}+{k_m-1 \choose 2}+(t_1+t_2-2){k_m-1 \choose 2}+{r_1 \choose 2}+{r_2 \choose 2}\\&\leq&{k_1+k_2+k_m-1 \choose 2}+(t_3-1){k_m-1 \choose 2}+{r_3 \choose 2}.
\end{eqnarray*}
If $k_1=0$ or $k_2=0$, similarly we can prove that $[n_1,k_1+k_m,k_m]+[n_2,k_2+k_m,k_m]<[n_1+n_2,k_1+k_2+k_m,k_m]$.

{\bf Observation 4:} Let $n_1\ge k_1+k_m $ and $n_2\ge k_2+k_m$. Then $$[n_1,k_1+k_m]+[n_2,k_2+k_m]<[n_1+n_2,k_1+k_2+k_m].$$

{\it Proof.} This observation follows from the following inequality:
\begin{eqnarray*}
[n_1,k_1+k_m]+[n_2,k_2+k_m]&
\leq&[n_1+n_2-\lfloor\frac{k_2+k_m-2}{2}\rfloor,k_1+k_m]
+e(K_{\lfloor\frac{k_2+k_m-2}{2}\rfloor})\\&<&[n_1+n_2,k_1+k_2+k_m].
\end{eqnarray*}

{\bf Observation 5:} Let $n_1\ge k_1+k_m$. Then $$[n_1,k_1+k_m]+[n_2,k_m,k_m]<[n_1+n_2,k_1+k_m].$$

{\it Proof.} Let $n_2=t_2(k_m-1)+r_2$, $$G_{1}=K_{\lfloor\frac{k_1+k_m-2}{2}\rfloor}+\overline{K}_{n-\lfloor\frac{k_1+k_m-2}{2}\rfloor} \mbox{ and }G_{2}=t_2\cdot K_{k_m-1}+K_{r_2}.$$ Since $e(G_2)<n_2(\lfloor\frac{k_1+k_m}{2}\rfloor-1)$, this observation follows easily.

{\bf Observation 6:} Let $n_1\ge k_1$ and $n_2\ge k_2+k_m$.  Then $$[n_1,k_1+k_m]+[n_2,k_2+k_m,k_m]<[n_1+n_2,k_1+k_2+k_m].$$

{\it Proof.} The proof of this observation is similar to the proof of observation 5.

{\bf Observation 7:} Let $n_1\ge k_1+k_m,n_2\ge k_2$. Then
 \begin{eqnarray*}
&& [n_1,k_1+k_m,k_m]+[n_2,k_2+k_m]\\&< &\max\left\{[n_1+n_2,k_1+k_2+k_m,k_m],[n_1+n_2,k_1+k_2+k_m]\right\}.
\end{eqnarray*}
{\it Proof.} We consider the following two cases.

{\bf Case 1}. If $n_2\geq\lfloor\frac{k_1}{2}\rfloor+\lfloor\frac{k_2}{2}\rfloor+\lfloor\frac{k_m}{2}\rfloor-1$, then $[n_1,k_1+k_m,k_m]+[n_2,k_2+k_m]<[n_2,k_1+k_2+k_m]+[n_1,k_1+k_m,k_m] \leq[n_1+n_2,k_1+k_2+k_m]$.\\

{\bf Case 2.} If $n_2\leq\lfloor\frac{k_1}{2}\rfloor+\lfloor\frac{k_2}{2}\rfloor+\lfloor\frac{k_m}{2}\rfloor-1$, then $[n_1,k_1+k_m,k_m]+[n_2,k_2+k_m]< [n_1+n_2,k_1+\lfloor\frac{k_2}{2}\rfloor+\lfloor\frac{k_m}{2}\rfloor+k_m,k_m]\leq[n_1+n_2,k_1+k_2+k_m,k_m]$.
\label{lastpage}
\end{document}